\documentclass[11pt]{amsart}

\usepackage{amsmath}
\usepackage{amssymb}
\usepackage[all]{xy}

\makeatletter
\@addtoreset{equation}{section}
\makeatother

\newtheorem{theorem}{Theorem}[section]
\newtheorem{proposition}[theorem]{Proposition}
\newtheorem{lemma}[theorem]{Lemma}
\newtheorem{corollary}[theorem]{Corollary}
\newtheorem{question}[theorem]{Question}

\theoremstyle{definition}

\newtheorem{remark}[theorem]{Remark}

\newcommand{\what}[1]{\widehat{#1}}
\newcommand{\wbar}[1]{\overline{#1}}
\newcommand{\til}[1]{\tilde{#1}}

\newcommand{\Cee}{\mathbb{C}}
\newcommand{\En}{\mathbb{N}}
\newcommand{\Ree}{\mathbb{R}}
\newcommand{\Tee}{\mathbb{T}}
\newcommand{\Zee}{\mathbb{Z}}

\newcommand{\alp}{\alpha}
\newcommand{\del}{\delta}

\newcommand{\sig}{\sigma}

\newcommand{\fA}{\mathcal{A}}
\newcommand{\fB}{\mathcal{B}}

\newcommand{\fH}{\mathcal{H}}

\newcommand{\fN}{\mathcal{N}}
\newcommand{\fT}{\mathcal{T}}
\newcommand{\fX}{\mathcal{X}}

\newcommand{\norm}[1]{\left\|#1\right\|}

\newcommand{\id}{\mathrm{id}}
\newcommand{\amen}{\mathrm{AM}}
\newcommand{\bl}{\mathrm{L}}
\newcommand{\fal}{\mathrm{A}}
\newcommand{\conj}{\mathrm{Conj}}
\newcommand{\cent}{\mathrm{Z}}
\newcommand{\za}{\mathrm{ZA}}
\newcommand{\vn}{\mathrm{VN}}
\newcommand{\im}{\mathrm{Im}}

\newcommand{\ideal}{\mathrm{I}}
\newcommand{\supp}{\mathrm{supp}}
\newcommand{\spn}{\mathrm{span}}
\newcommand{\spo}{\mathrm{SO}}
\newcommand{\spu}{\mathrm{SU}}

\begin{document}

\title[Amenability properties of $\za(G)$]
{Amenability properties of the central Fourier algebra of a compact group}

\author{Mahmood Alaghmandan and Nico Spronk}

\begin{abstract}
We let the central Fourier algebra, $\za(G)$, be the subalgebra of functions
$u$ in the Fourier algebra  $\fal(G)$ of a compact group, for which
$u(xyx^{-1})=u(y)$ for all $x,y$ in $G$.  We show that this algebra admits bounded point derivations
whenever $G$ contains a non-abelian closed connected subgroup.  Conversely when
$G$ is virtually abelian, then $\za(G)$ is amenable.  Furthermore, for virtually abelian $G$,
we establish which closed ideals admit bounded approximate identities.
We also show that if $\za(G)$ is weakly amenable, in fact hyper-Tauberian,
exactly when $G$ admits no non-abelian connected subgroup.  
We also study the amenability constant of $\za(G)$ for finite $G$
and exhibit totally disconnected groups $G$ for which $\za(G)$ is non-amenable.
In passing, we establish some properties related to spectral synthesis of subsets
of the spectrum of $\za(G)$.
\end{abstract}

\maketitle

\footnote{{\it Date}: \today.

}

Let $G$ be a compact group with Fourier algebra $\fal(G)$
and $B$ be a group of continuous automorphisms on $G$.  We let
\[
\cent_B\fal(G)=\bigcap_{\beta\in B}\{u\in\fal(G):u\circ\beta=u\}
\]
and call this the algebra the {\it $B$-centre} of $\fal(G)$.  In particular
we let $\mathrm{Inn}(G)$ be the group of inner automorphisms and let
\[
\za(G)=\cent_{\mathrm{Inn}(G)}\fal(G)
\]
and call this the {\it $G$-centre} of $\fal(G)$, or the {\it central Fourier algebra}
on $G$.  Of course, since $\fal(G)$
is a  commutative algebra, this bears no relation to the centre of $\fal(G)$ as an algebra.

\subsection{Motivation and plan}
We are motivated by the results of \cite{azimifardss}, where amenability
properties of the centre of the group algebra, $\cent\bl^1(G)$, are studied.  We note that
this algebra is densely spanned by idempotents --- i.e.\ normalized characters
of  irreducible representations, $d_\pi\chi_\pi$ --- and hence is automatically weakly amenable,
even hyper-Tauberian (see \cite[Theo.\ 14]{samei}).  Hence those properties were not discussed
for compact $G$.
However, it was shown for $G$ either non-abelian and connected, or an infinite product of 
non-abelian finite groups, that $\cent\bl^1(G)$ is non-amenable.  It is conjectured in that
paper that $\cent\bl^1(G)$ is amenable if and only if if $G$ has an open abelian subgroup.
Neither direction of that conjecture is resolved.

In the present article, we conduct a parallel investigation for $\za(G)$.  Our techniques
are different and some of our conclusions sharper.  In the case that the connected component
of the identity, $G_e$, is non-abelian, we show that $\za(G)$ admits a point derivation;
see Section \ref{sec:pointder}.  In Section \ref{sec:openabel} we show that when $G$
is virtually abelian then $\za(G)$ is amenable.  We are able
to use the structure of the central Fourier algebra for virtually abelian compact groups, 
to further establish that $\za(G)$
is weakly amenable if and only if $G_e$ is abelian.  We invest the extra effort to establish
that this is exactly the case in which $G$ is hyper-Tauberian, a condition identified and studied by
Samei (\cite{samei}).  One of our major tools in this section is the relationship between
these properties and certain conditions related to sets of spectral synthesis.
In Section \ref{sec:amenconst} we investigate the amenability constant 
$\amen(\za(G))$, and show that an infinite product of non-abelian finite groups, $P$
gives a non-amenble algebra $\za(P)$.

In the course of our investigation we gain, in Proposition \ref{prop:finitesynthweaksynth},
some results on spectral synthesis and weak synthesis
of singleton and finite subsets of the spectrum of $\za(G)$, giving partial generlisations of results
of Meany \cite{meaney} and Ricci \cite{ricci}, with quite different proofs.  
We also gain a broad generalisation of an amenability result
of Lasser \cite{lasser}; see Remark \ref{rem:lasser}. We also show in Propositions
\ref{prop:virtabel1} and \ref{prop:virtabel2}, for virtually abelian
compact groups, a full description of ideals in $\za(G)$ admitting bounded approximate identities

\subsection{History}  
It was shown by Johnson (\cite{johnson94}) that $\fal(G)$ may fail to be amenable for a compact $G$.
This led Ruan (\cite{ruan}) to define operator amenability and show that $\fal(G)$
is operator amenable exactly when $G$ is amenable; in particular this holds when $G$ is compact.
One of the most interesting applications of this result is establishing, for amenable $G$,
which closed ideals of $\fal(G)$ admit bounded approximate identites, due to
Forrest, Kaniuth, Lau and Spronk (\cite{forrestkls}).  For abelian $G$, this was established by
Liu, Van Rooij and Wang in \cite{liuvrw}.
That $\fal(G)$ is always operator weakly amenable
was established by Spronk (\cite{spronk0}), and, independantly Samei (\cite{samei0}).  
Both articles relied on identification of sets of spectral synthesis, and, in particular, ideas in the
latter latter helped Samei establish the notion of hyper-Tauberian Banach algebras in \cite{samei}.
The operator space structure of $\za(G)$ is maximal -- see \ref{ssec:zag} -- and hence operator versions
of amenability propeties are automatic.

Forrest and Runde (\cite{forrestr}) established that $\fal(G)$ is amenable exactly when $G$
is virtually abelian, and further that if the connected compenent $G_e$ is abelian, then
$\fal(G)$ is weakly amenable.  Recently, Lee, Ludwig, Samei and Spronk (\cite{leelss})
have established for a Lie group $G$, then $\fal(G)$ can be weakly amenable only
when $G_e$ is abelian.  For compact $G$, this fact was established by
Forrest, Samei and Spronk in \cite{forrestss}.

\section{Preliminaries and notation.}\label{sec:preliminaries}

\subsection{Amenability and weak amenability}
We briefly note some fundamental definitions which go back to \cite{johnson}.

Let $\fA$ be a commutative Banach algebra.  A Banach $\fA$-bimodule is any
Banach space $\fX$ which admits a pair of contractive homomorphisms $a\mapsto (x\mapsto ax)$
and $a\mapsto (x\mapsto xa)$, each from $\fA$ into
bounded operators $\fB(\fX)$, such that the ranges of these maps commute:  $a(xb)=(ax)b$.
The homomorphisms given by pointwise adjoints of these maps make the dual, $\fX^*$, into
a {\it dual Banach $\fA$-module}.  We say $\fA$ is {\it amenable} is every bounded derivation
from $\fA$ into any dual module, i.e.\ $D:\fA\to\fX^*$ with $D(ab)=a(Db)+(Da)b$, is inner, i.e.\
$Da=af-fa$ for some $f$ in $\fX^*$.

Following \cite{badecd}, we say $\fA$ is {\it weakly amenable} if there are no non-zero
bounded derivations for $\fA$ into any {\it symmetric} Banach $\fA$-module, i.e.\ module $\fX$
satisfying $ax=xa$.  In particular, $\fA$ is a symmetric Banach $\fA$-bimodule as is its dual $\fA^*$.
It is sufficient to see that there are no non-zero bounded derivations from $\fA$ into $\fA^*$ to show
that $\fA$ is weakly amenable.  Given a multiplicative functional $\chi$ on $\fA$, a {\it point derivation}
at $\chi$ is any functional $D$ on $\fA$ which satisfies $D(ab)=\chi(a)D(b)+D(a)\chi(b)$.
The map $a\mapsto D(a)\chi:\fA\to\fA^*$ is then a derivation.   Hence a weakly amenable 
algebra admits no bounded point derivations.

\subsection{The central Fourier algebra}\label{ssec:zag}
Throughout this article, with the exception of Section \ref{ssec:virtabelian},
$G$ will denote a compact group.  We let $\what{G}$ denote the set
of (equivalence classes of) continuous irreducible representations.  For $\pi$ in $\what{G}$
we let $\fH_\pi$ denote the space on which it acts and let $d_\pi=\dim\fH_\pi$.  We denote
normalized Haar integration by $\int_G \dots ds$.  For integrable $u:G\to\Cee$ we let
$\hat{u}(\pi)=\int_G u(s)\bar{\pi}(s)\,ds$, and \cite[(34.4)]{hewittrII} provides the following
description of the {\it Fourier algebra}:
\[
u\in\fal(G)\quad\Leftrightarrow\quad
\norm{u}_\fal=\sum_{\pi\in\what{G}}d_\pi\norm{\hat{u}(\pi)}_1<\infty
\]
where $\norm{\cdot}_1$ denotes the trace norm.
This is a special case of the general definition of the Fourier algebra, for locally compact groups, 
given in \cite{eymard}.
We let $\vn(G)=\ell^\infty\text{-}\bigoplus_{\pi\in\what{G}}\fB(\fH_\pi)$ and we have dual identification
$\fal(G)^*\cong\vn(G)$ via
\[
\langle u,T\rangle=\sum_{\pi\in\what{G}}d_\pi\mathrm{Tr}(\hat{u}(\pi)T_\pi).
\]

Let $Z_G:\fal(G)\to\za(G)$ be given by $Z_Gu(x)=\int_G u(sxs^{-1})\, ds$.  This is easily
verified to be a contractive linear projection.  A well-known
consequence of the Schur orthogonality relations is that for $\pi$ in $\what{G}$ and
$\xi,\eta$ in $\fH_\pi$ we have $Z_G\langle\pi(\cdot)\xi|\eta\rangle
=\frac{\langle\xi|\eta\rangle}{d_\pi}\chi_\pi$, where $\chi_\pi$ is the character of 
$\pi$. Hence $\za(G)=\wbar{\spn}\{\chi_\pi:\pi\in\what{G}\}$ .  
Thus, since $\hat{\chi}_\pi(\pi)=\frac{1}{d_\pi}I_{d_\pi}$ and $\hat{\chi}_\pi(\pi')=0$ for $\pi'\not=\pi$,
we have $\norm{\hat{\chi}_\pi(\pi)}_1=1$, and the
description of the norm above gives that
\begin{equation}\label{eq:zanorm}
u\in\za(G)\quad\Leftrightarrow\quad u=\sum_{\pi\in\what{G}}\alp_\pi\chi_\pi\text{ with }
\norm{u}_\fal=\sum_{\pi\in\what{G}}d_\pi|\alp_\pi|<\infty.
\end{equation}
Hence we have that $\za(G)^*\cong\cent\vn(G)=\ell^\infty\text{-}\bigoplus_{\pi\in\what{G}}\Cee I_{d_\pi}
\cong\ell^\infty(\what{G})$.  

In particular, we see that $\za(G)$ is the predual of a commutative von
Neumann algebra.  Thus, generally, we will have no need to discuss the completely bounded 
theory of this space.
However, we shall require, some knowledge of the operator space structure on $\fal(G)$
in Section \ref{sec:amenconst}, which we shall simply reference therein.

Let $\sim_G$ denote the equivalence relation on $G$ by conjugacy, i.e.\
$x\sim_G x'$ if and only if $x'=yxy^{-1}$ for some $y$ in $G$.  It may be deduced from
\cite[(34.37)]{hewittrII} that each element of the Gelfand spectrum of $\za(G)$
is given by evaluation functionals from $G$, and hence may be identified with
a point in $\conj(G)=G/{\sim_G}$.  See also Proposition
\ref{prop:spec}, which gives a generalisation of this result.
It is evident that $Z_G$ has a certain expectation property: $Z_G(uv)=uZ_Gv$ for $u$ in
$\za(G)$ and $v$ in $\fal$.

We observe that $\za(G)$ is actually the {\it hypergroup algebra} $\ell^1(\what{G},d^2)$.
Indeed if we let $\del_\pi=\frac{1}{d_\pi}\chi_\pi$ then we obtain the hypergroup convolution
$\del_\pi\del_{\pi'}=\sum_{\sig\subset\pi\otimes\pi'}\frac{d_\sig}{d_\pi d_{\pi'}}\del_\sig$, where the sum
is given all over irreducible subrepresentations of $\pi\otimes\pi'$, allowing repetitions for multiplicity.
Here the Haar measure is given by $d^2(\{\pi\})=d^2_\pi$.  Notice that $\cent\bl^1(G)$, the algebra
studied in \cite{azimifardss}, is the hypergroup algebra associated with the compact hypergroup
$\conj(G)$.  Moreover, $\what{G}$ and $\conj(G)$ are mutually dual hypergroups.  See 
\cite{jewett} to see that $\what{G}$ is the dual of $\conj(G)$.  That $\conj(G)$ is the dual of
$\what{G}$ was stated in \cite{azimifardss}, but without reference nor proof.  A proof of this fact
will be given by Amini and the first-named author in a forthcoming article.  It combines the facts
that characters on $\za(G)\cong\ell^1(\what{G},d^2)$ are in bijective correspondence with elements 
of the spectrum of $\za(G)$.

Let us end this section with a simple observation on quotient groups.

\begin{proposition}\label{prop:quotient}
Let $N$ be a closed normal subgroup of $G$.  Then $P_Nu(x)=\int_H u(xn)\,dn$ (normalized
Haar integration on $N$) defines a surjective quotient map from $\za(G)$ to $\za(G/N)$.
\end{proposition}

\begin{proof}
Since translation is an isometric action on $\fal(G)$, continuous in $G$, it is
clear that $P_N:\za(G)\to\fal(G/N)$ is a contraction.
Let $q:G\to G/N$ be the quotient map.  Then there is an obvious embedding
$\pi\mapsto\pi\circ q:\what{G/N}\to\what{G}$.  An easy calculation shows that
for $\pi'$ in $\what{G}$ we have
\[
P_N\chi_{\pi'}=\begin{cases} \chi_\pi\circ q &\text{if }\pi'=\pi\circ q \\
0 &\text{otherwise.}\end{cases}
\]
See, for example, the proof of \cite[Prop.\ 4.14]{ludwigst}.  Further, 
an application of (\ref{eq:zanorm}) shows that the map
\[
\sum_{\pi\in\what{G/N}}\alp_\pi\chi_\pi\mapsto\sum_{\pi\in\what{G/N}}\alp_\pi\chi_\pi\circ q:
\za(G/N)\to\za(G)
\]
is an isometry whose range is the image of $P_N$.
\end{proof}

\section{Groups with non-abelian connected component}\label{sec:pointder}

We begin with the case of $G=\spu(2)$, the group of $2\times 2$ unitary matrices
of determinant one.  We recall that the conjugacy classes in $\spu(2)$ are determined
by the eigenvlues; i.e.\ for each $x$ in $\spu(2)$ there
is $y$ in $\spu(2)$ for which
\begin{equation}\label{eq:diag}
x=y\begin{bmatrix} \zeta & 0 \\ 0 & \zeta^{-1} \end{bmatrix}y^{-1}\text{ where }
\zeta\text{ in }\Tee\text{ satisfies }\im\zeta\geq 0.
\end{equation}
Hence we may label the conjugacy class by the specified eigenvalue: $C_x=C_\zeta$,
so $C_\zeta=C_{\zeta^{-1}}$.   It is well known that
\[
\what{\spu}(2)=\{\pi_l:l=0,1,2,\dots\}
\]
where $d_{\pi_l}=l+1$ and we have the associated character, evaluated at $x$ as in
(\ref{eq:diag}), given by
\[
\chi_l(x)=\chi_{\pi_l}(C_x)=\chi_{\pi_l}(C_\zeta)
=\sum_{k=0}^l\zeta^{l-2k}=\begin{cases}
\frac{\zeta^{l+1}-\zeta^{-l-1}}{\zeta-\zeta^{-1}} &\text{if }\zeta\in\Tee\setminus\{-1,1\} \\
\zeta^l(l+1) &\text{if }\zeta\in\{-1,1\}.\end{cases}
\]

We now recall that the $3\times 3$ special orthogonal group $\spo(3)$ is isomorphic
to $\spu(2)/\{-I,I\}$ and admits spectrum
$\what{\spo}(3)=\{\pi_l:l=0,2,4,\dots\}\subset\what{\spu}(2).$
The following is a sort of refinement of a result in \cite{chilanar}, which is adapted specifically
for our proof of Theorem \ref{theo:nacc}.

\begin{proposition}\label{prop:suso}
Let $S$ denote either $\spu(2)$ or $\spo(3)$.  
Then $\za(S)$ admits a non-zero point derivation $D_z$ at each class $C_z$ of $S$
for which $z$ is a transcendental element of $\Tee$.
\end{proposition}

\begin{proof}
Let $\cent\fT=\spn\{\chi_l:l=0,1,2,\dots\}$ which is dense in $\za(\spu(2))$.
For $u$ in $\cent\fT$ and $z$ in $\Tee$ with $\im z>0$ let 
$D_zu=\left.z\frac{d}{d\zeta}u(C_\zeta)\right|_{\zeta=z}$.  For such $z$ we compute
\[
D_z\chi_l
=\frac{l(z^{l+2}-z^{-l-2})-(l+2)(z^l-z^{-l})}{(z-z^{-1})^2}
\]
This implies that
\[
|D_z\chi_l|\leq\frac{4l+4}{|z-z^{-1}|^2}.   
\]
Moreover, if $z$ is transcendental, then
$D_z\chi_l\not=0$ for any $l>0$.  
Now if $u=\sum_{j=1}^n\alp_j\chi_{l_j}$ for $l_1<\dots<l_n$, we can use the formula
for the norm (\ref{eq:zanorm}) to see that
\[
|D_zu|\leq\sum_{j=1}^n|\alp_j|\frac{4l_j+4}{|z-z^{-1}|^2}=\frac{4}{|z-z^{-1}|^2}\sum_{j=1}^n|\alp_j|(l_j+1)
=\frac{4}{|z-z^{-1}|^2}\norm{u}_\fal.
\]
Hence $D_z$ extends to a continuous derivation on $\za(\spu(2))$.

By virtue of Proposition \ref{prop:quotient} and the identification $\what{\spo}(3)
\subset\what{\spu}(2)$, above, we have
\[
\za(\spo(3))=\wbar{\spn}\{\chi_l:0,2,4,\dots\}\,\til{\subset}\,\za(\spu(2)).
\]
Hence each derivation $D_z$, with transcendental $z$, also defines a non-zero point derivation
on $\za(\spo(3))$.
\end{proof}


\begin{theorem}\label{theo:nacc}
Let $G$ have non-abelian connected component $G_e$.  Then $\za(G)$ admits a non-zero
point derivation.
\end{theorem}

\begin{proof}
According to the proof of \cite[Theo.\ 2.1]{forrestss}, $G_e$ admits a closed subgroup $S$
which is isomorphic to $\spu(2)$ or $\spo(3)$.  [This uses a structure theorem for connected
compact groups --- see \cite{price} --- and the fact that any compact non-abelian Lie algebra
admits a copy of $\mathfrak{su}(2)$.]  Since $\za(G)|_S\subset\fal(G)|_S=\fal(S)$ (see
\cite{herz} or \cite[(34.27)]{hewittrII}, for example), and since for $u$ in $\za(G)$, $u(yxy^{-1})=u(x)$
for $x,y$ in $G$, a fortiori in $S$, we have $\za(G)|_S\subseteq\za(S)$.  Further, since
$S$ is connected and $S\supsetneq\{e\}$, $S$ is contained in more than one conjugacy
class of $G$.  Hence $\za(G)|_S\not\subset\Cee1$, and there is $\pi$ in $\what{G}$ for which
$\chi_\pi|_S\not\in\Cee 1$.  Thus we have
\[
\pi|_S=\bigoplus_{j=1}^nm_j\pi_{l_j}
\]
where each $m_j$ is a non-zero multiplicity and $l_1<\dots<l_n$ with either $n>1$ or $l_1>0$.
It follows that $\chi_\pi|_S=\sum_{j=1}^nm_j\chi_{l_j}$ so
\[
\chi_\pi(C_\zeta)=\sum_{j=1}^nm_j\sum_{k_j=0}^{l_j}\zeta^{l_j-2k_j}
\]
where $C_\zeta$ is the conjugacy class of elements with eigenvelue $\zeta$
in $S$.  Thus for $z$ in $\Tee$ with $\im z>0$ we have for the derivation $D_z$, defined in
the proposition above, that
\[
D_z(\chi_\pi|_S)=\sum_{j=1}^n\sum_{k_j=0}^{l_j}m_j(l_j-2k_j)z^{l_j-2k_j}.
\]
Notice that the above expression is a non-zero polynomial in $z$ of degree $l_n$.
If $z$ is transcendental over rationals, then $D_z(\chi_\pi|_S)\not=0$.
Thus for such $z$, $D=D_z\circ R_S:\za(G)\to\Cee$, where $R_S$ is the restriction map,
is a non-zero point derivation.
\end{proof}

\begin{remark}\label{rem:nacc}
{\bf (i)}  For semisimple compact Lie $G$, a simplification of a result in \cite{ricci} gives bounded
point derivations at all regular points of $G$.  This offers more precise data than
does our Theorem \ref{theo:nacc} for such groups.

{\bf (ii)}  As mentioned in Section \ref{sec:preliminaries}, $\za(G)\cong\ell^1(\what{G},d^2)$
is a discrete hypergroup algebra.
There are other discrete hypergroups, namely certain polynomial hypergroups, whose
hypergroup algebras are known to admit point derivations.  See \cite{lasser}.

{\bf (iii)}  There are other known examples of $B$-central Fourier algebras which admit point derivations.
For example, if $n\geq 3$, then the algebra of radial elements of $\fal(\Ree^n)$,
$\cent_{\spo(n)}\fal(\Ree^n)$, admits a point derivation at each infinite orbit (\cite[2.6.10]{stegemanr}).  
We remark that if we let $H_n=\Ree^n\rtimes\spo(n)_d$, $\Ree^n$ acted upon
by the discretized special orthogonal group, then for odd $n$ we have
$\za(H_n)=\cent_{\spo(n)}\fal(\Ree^n)$, while for even $n$, 
$\za(H_n)=\cent_{\spo(n)/\{\pm I\}}\fal(\Ree^n\rtimes\{\pm I\})$.

A general analysis of algebras $\za(H)$ for locally compact $H$ is beyond the scope of our present
investigation.  For groups with pre-compact conjugacy classes (a class which does not include 
examples $H_n$, above), there are some results for $\cent\bl^1(H)$ in \cite{azimifardss}.
\end{remark}

\section{Virtually abelian groups and groups with abelian connected components}\label{sec:openabel}

\subsection{On sets of synthesis in certain fixed-point subalgebras}\label{ssec:zba}
The purpose of this section is to gather some abstract results which will be useful for understanding
$\za(G)$ for a virtually abelian (locally) compact group $G$, in the next section.

For this section we let $\fA$ denote a commutative, unital, semisimple Banach algebra which is regular
in its Gelfand spectrum $X$.  Given a group of automorphisms $B$ on $\fA$ we let
\[
\cent_B\fA=\bigcap_{\beta\in B}\{u\in\fA:\beta(u)=u\}
\]
denote the fixed point algebra.  If $B$ is compact and acts continuously, we let $Z_B:\fA\to\cent_B\fA$
be given by $Z_Bu=\int_B \beta(u)\,d\beta$ (normalized Haar measure), which
may be understood as a Bochner integral.  It is a surjective quotient map which satisfies the expectation
property $Z_B(uv)=uZ_Bv$ for $u$ in $\cent_B\fA$ and $v$ in $\fA$.

\begin{proposition}\label{prop:spec}
Let $B$ be a compact group of continuous automorphisms on $\fA$.
The Gelfand spectrum of $\cent_B\fA$ is the orbit space $X/B$.
\end{proposition}

\begin{proof}
Since $X$ is the spectrum of $\fA$, each $\beta$ in $B$ defines an automorphism $\beta^*|_X$
of $X$.  It is clear that the orbit space $X/B=\{B^*x:x\in X\}$, with quotient topology comprises a 
closed subset of the spectrum of $\cent_B\fA$.  The regularity of $\fA$ passes immediately
to the regularity of $\cent_B\fA$ on $X/B$.  Indeed, if $B^*x\not=B^*x'$, there is
$u$ in $\fA$ for which $u|_{B^*x}=1$ and $u|_{B^*x'}=0$.   

Let $\chi$ be any multiplicative character on $\cent_B\fA$.
Suppose that for some $u_1,\dots, u_n$ in $\ker\chi$, $\bigcap_{k=1}^nu_k^{-1}\{0\}\cap
X/B=\varnothing$.  Then for any $x$ in $X$, i.e.\ $B^*x$ in $X/B$, let 
$z_{xk}=\wbar{u_k(x)}$, for $k=1,\dots,n$, so
\[
\sum_{k=1}^n z_{xk}u_k(x')>0\text{ for all }x'\text{ in a neighbourhood in }X/B,\;U_x\text{ of }x.
\]
Let $U_{x_1},\dots, U_{x_m}$ be a finite subcover of $X/B$,  By regularity --- see, for example,
the proof of \cite[(39.21)]{hewittrII} --- there is a partition of unity $w_1,\dots,w_m$ in $\cent_B\fA$
subordinate to the cover $U_{x_1},\dots,U_{x_m}$.  Then
\[
u=\sum_{j=1}^mw_j\sum_{k=1}^nz_{x_jk}u_k\in\ker\chi
\]
and is pointwise positive, hence non-vanishing on $X$. Thus $u$ admits
an inverse $u'$ in $\fA$.  But then $Z_Bu'$ is the inverse of $u$ in $\cent_B\fA$, contradictiong that
$u\in\ker\chi$.
We thus conclude that for any finite family $F\subset\cent_B\fA$, $\bigcap_{u\in F}u^{-1}\{0\}\cap
X/B\not=\varnothing$, and a compactness argument yields that $\bigcap_{u\in\cent_B\fA}u^{-1}\{0\}\cap
X/B\not=\varnothing$.  It follows that $\chi\in X/B$.
\end{proof}

\begin{remark} {\bf (i)}
 If we assume that $\fA$ is conjugate-closed, qua functions on $X$, then for a family
of elements $u_1,\dots, u_n$ for which $\bigcap_{k=1}^nu_k^{-1}\{0\}\cap X/B=\varnothing$, we have
that  $u=\sum_{k=1}^n|u_k|^2$ forms the invertible element in the proof above.  This would allow us 
to circumvent  the partition of unity argument.  See \cite[(34.37)]{hewittrII}.

{\bf (ii)} We can recover the result of \cite{forrest0} that for a compact subgroup $K$ of a locally compact
group $H$, the algebra $\fal(H:K)=\{u\in\fal(H):u(xk)=u(x)\text{ for }u\text{ in }H\text{ and }k\text{ in }K\}$
has spectrum the coset space $G/K$.  Indeed, consider the unitization $\fal(H)\oplus\Cee 1$
and let $K$ act as automorphisms on this algebra by right translation; we obtain $\fal(H:K)$
as a the subalgebra of $\cent_K(\fal(H)\oplus\Cee 1)$ of functions vanishing at infinity.

\end{remark}

Now let $E$ be a closed subset of $X$.  We let
\[
\ideal_\fA(E)=\{u\in\fA:u|_E=0\},\text{ and }
\ideal^0_\fA(E)=\{u\in\fA:\supp u\cap\ E=\varnothing\}.
\]
We say that $E$ is 

$\bullet$ {\it spectral} for $\fA$ if $\wbar{\ideal^0_\fA(E)}=\ideal_\fA(E)$;

$\bullet$ {\it Ditkin} for $\fA$ provided each $u$ in $\ideal_\fA(E)$ satisfies
that $u\in\wbar{u\ideal^0_\fA(E)}$;

$\bullet$ {\it hyper-Ditkin} if $\ideal_\fA^0(E)$ posseses a bounded approximate identity for 
$\ideal_\fA(E)$; and

$\bullet$ {\it approximable} if $\ideal_\fA(E)$ posseses a bounded approximate identity.

\noindent Following \cite{wik}, we will say that $E$ is {\it strongly Ditkin} for $\fA$ if
$\ideal_\fA^0(E)$ posseses a multiplier bounded approximate identity $(u_\alp)$ or $\ideal_\fA(E)$,
i.e.\ so $\sup_\alp\norm{u_\alp u}\leq C\norm{u}$ for all $u$ in $\ideal_\fA(E)$.
Note that a sequential approximate identity is automatically multiplier bounded, thanks to the
uniform boundedness principle.  We shall not require this last notion but mention it only for
comparative purposes.  We have the following implications of properties for a closed subset $E$ of $X$:
\[
\begin{matrix} \text{approximable} \\ \text{\& spectral}\end{matrix}
\;\Rightarrow\;
\text{hyper-Ditkin} 
\;\Rightarrow\;
\begin{matrix} \text{strongly} \\ \text{Ditkin} \end{matrix}
\;\Rightarrow\;
\text{Ditkin}
\;\Rightarrow\;
\text{spectral.}
\]

\begin{remark}\label{rem:belcherr}
It is not the case that approximable implies spectral.  In \cite{blecherr}, a 
remarkable example of a regular squence algebra
$\fA$ is created which admits a contractive approximate identity, but for which the space finitely
supported elements $\fA_c$ is not dense in $\fA$.  In particular, the unitization
$\fA\oplus\Cee 1$ has spectrum the compactification $\En\cup\{\infty\}$, and
hence $\{\infty\}$ is an approximable but non-spectral set for $\fA\oplus\Cee 1$.
\end{remark}

\begin{remark}\label{rem:wik}
A longstanding open question is whether finite unions of spectral sets are spectral.
It is known, due to \cite{warner0} (see also \cite[5.2.1]{kaniuthB}) that a finite union of Ditkin sets is Ditkin.
An easy variant of an argument of \cite{wik} tells us that the same is true for hyper-Ditkin or approximable
sets.  Indeed, if $E$ and $F$ are approximable (respectively, hyper-Ditkin), let
$(u_\alp)$ be a bounded approximate identity for $\ideal_\fA(E)$ and
$(v_\beta)$ one for $\ideal_\fA(F)$ (each contained in $\ideal_\fA^0(E)$, respectively
$\ideal_\fA^0(F)$).  Then $(u_\alp v_\beta)$ (product directed set) is a bounded approximate identity for
$\ideal_\fA(E\cup F)=\ideal_\fA(E)\cap\ideal_\fA(F)$ (contained in $\ideal^0_\fA(E\cup F)=
\ideal_\fA^0(E)\cap\ideal_\fA^0(F)$), as is easily checked.
\end{remark}

We now observe that finite groups of automorphisms preserve certain properties of sets which
are stable under finite unions.

\begin{theorem}\label{theo:zbafinite}
Let $B$ be a finite group of automophisms on $\fA$ and $E$ be a closed subset of $X$
which is Ditkin, hyper-Ditkin or approximable for $\fA$.  Then the subset $B^*E$
of $X/B$ enjoys the same property for $\cent_B\fA$.
\end{theorem}

\begin{proof}
We observe that for each automorphism $\beta$, we have $\beta(\ideal_\fA(E))=\ideal_\fA(\beta^*E)$
and $\beta(\ideal_\fA^0(E))=\ideal_\fA^0(\beta^*E)$.  Hence Remark \ref{rem:wik} shows that
$B^*E=\bigcup_{\beta\in B}\beta^*E$ is Ditkin, hyper-Ditkin or approximable for $\fA$, based
on the respective assumption for $E$.  Furthermore, we observe that for any $u$ in $\fA$ we have
$Z_Bu=\frac{1}{|B|}\sum_{\beta\in B}\beta(u)$, and hence
\[
Z_B\ideal^0_\fA(B^*E)=\ideal_{\cent_B\fA}^0(B^*E)\subseteq\ideal^0_\fA(B^*E)
\]
and the same sequence of inclusions holds for $\ideal_\fA$.
Suppose $u\in\ideal_{\cent_B\fA}(B^*E)$ and $(u_\alp)$ is a net
from $\ideal_\fA^0(E)$ for which $\norm{uu_\alp-u}\overset{\alp}{\longrightarrow}0$.
Then  $uZ_Bu_\alp-u=Z_B(uu_\alp-u)\overset{\alp}{\longrightarrow}0$.  We immediately
see that Ditkenness or hyper-Ditkenness is preserved.  By merely picking $(u_\alp)$
from within $\ideal_\fA(E)$, we see that approximability is preserved.
\end{proof}

\begin{proposition}\label{prop:zbatp}
If the projective tensor product $\fA\hat{\otimes}\fA$ is semisimple, and $B$ is a compact
group of automorphisms on $\fA$, then $\cent_B\fA\hat{\otimes}\cent_B\fA=
\cent_{B\times B}\fA\hat{\otimes}\fA$.
\end{proposition}

\begin{proof}
Since $Z_B$ is a quotient map, $\cent_B\fA\hat{\otimes}\cent_B\fA$ is isometrically
a subspace of $\fA\hat{\otimes}\fA$.  Moreover, $Z_B\otimes Z_B=Z_{B\times B}$.
\end{proof}

Suppose $\fA\hat{\otimes}\fA$ is semisimple.  With our assumptions $\fA\hat{\otimes}\fA$
is regular on its spectrum $X\times X$ (\cite{tomiyama}).
Then, following \cite[Theo.\ 6]{samei}, we call $\fA$ {\it hyper-Tauberian} if the diagonal
\[
X_D=\{(x,x):x\in X\}
\]
is spectral for $\fA\hat{\otimes}\fA$.  It is a well-known interpretation of the splitting
result of \cite{helemskii} (see also \cite{curtisl}) that approximability of $X_D$ for
$\fA\hat{\otimes}\fA$ is equivalent to amenability of $\fA$.  We further note that
for a compact group of automorphisms on $\fA$ that
\[
(X/B)_D=\{(B^*x,B^*x):x\in X\}=(B\times B)^*X_D.
\]
These comments combine with the last two results to give us the following.

\begin{corollary}\label{cor:zbafinaht}
Suppose $\fA\hat{\otimes}\fA$ is semisimple and let $B$ be a finite group of continuous automorphisms 
on $\fA$.

{\bf (i)}  If $X_D$ is Ditkin for $\fA\hat{\otimes}\fA$, then $\cent_B\fA$ is hyper-Tauberian.

{\bf (ii)} If $\fA$ is amenable, then $\cent_B\fA$ is amenable.
\end{corollary}

We observe that (ii), above, follows from a more general result of \cite{kepert}. 

We shall say that a closed subset $E$ of $X$ is {\it weakly spectral} for $\fA$ if there is
a fixed $n>0$ for which $\ideal_\fA(E)^n=\{u^n:u\in\ideal_\fA(E)\}\subseteq\wbar{\ideal^0_\fA(E)}$.
We let the characteristic of $E$ with respect to $\fA$, $\xi_\fA(E)$, denote the minimal such $n$, 
so $\xi(E)=1$ if $E$ is spectral.  These concepts were introduced in \cite{warner}.

\begin{proposition}\label{prop:weaksyn}
Suppose $B$ is a compact group of automorphisms on $\fA$ and $E=B^*E$ be weakly spectral
for $\fA$.  Then $E$ is weakly spectral for $\cent_B\fA$ (with $\xi_{\cent_B\fA}(E)\leq\xi_\fA(E)$).
\end{proposition}

\begin{proof}
It is evident that $Z_B\ideal_\fA(E)=\ideal_{\cent_B\fA}(E)$.  It is also true that
$Z_B\ideal_\fA^0(E)=\ideal_{\cent_B\fA}^0(E)$.  Indeed, if $u\in\ideal_\fA^0(E)$,
then $\supp u\cap E=\varnothing$.  Then there there is open $U=B^*U\supset E$
such that $\supp u\cap\wbar{U}=\varnothing$.  If not, then
\[
\supp u\cap E=\supp u\cap\bigcap\{\wbar{U}:U=B^*U\text{ open, }U\supset E\}\not=\varnothing
\]
violating our initial assumption.  Thus it follows that $\supp(Z_Gu)\cap E=\varnothing$.

We note that $\ideal_{\cent_B\fA}(E)\subset\ideal_\fA(E)$.
Hence if $E$ is weakly spectral for $\fA$, then for $u\in\ideal_{\cent_B\fA}(E)$,
$u^{\xi_\fA(E)}\in\wbar{\ideal_\fA^0(E)}$.  But 
\[
u^{\xi_\fA(E)}=Z_Bu^{\xi_\fA(E)}
\in Z_B\wbar{\ideal_\fA^0(E)}\subseteq\wbar{Z_B\ideal_\fA^0(E)}=\wbar{\ideal_{\cent_B\fA}^0(E)}.
\]
Hence $\xi_{\cent_B\fA}(E)\leq\xi_\fA(E)$.
\end{proof}

\subsection{Virtually abelian groups}\label{ssec:virtabelian}
We say a locally comapct group is virtually abelian if it admits an abelian subgroup of finite index.  In the case
of a compact $G$, this is equivalent to having an open abelian subgroup.

\begin{theorem}\label{theo:va}
Let $G$ be virtually abelian.  Then there exists a normal open abelian subgroup $T$.
We then have that the $T$-centre, i.e.\ when $T$ acts on $G$ by inner automorphisms, is given by
an isomorphic identification
\[
\cent_T\fal(G)=\bigoplus_{aT\in G/T}\fal(aT:R_a)
\]
where $R_a=R_{aT}$ is a closed subgroup of $T$ and $\fal(aT:R_a)=\{u\in1_{aT}\fal(G):
u(atr)=u(at)\text{ for }t\text{ in }T\text{ and }r\text{ in }R_a\}$.
The algebra $\cent_T\fal(G)$ admits spectrum $X=\bigsqcup_{aT\in G/T}aT/R_a$.
We have
\[
\za(G)=\cent_{G/T}\cent_T\fal(G)
\]
where the action of $G$, i.e.\ of $G/T$, on an element of $X$ is given by 
$bT\cdot atR_a=bab^{-1}btb^{-1}R_{bab^{-1}}$.  
For each $a$ in $G$ and $t$ in $T$, we have conjugacy class
\[
C_{at}=\{bab^{-1}btb^{-1}r:b\in G\text{ and }r\in R_{bab^{-1}T}\}.
\]
\end{theorem}

\begin{proof}
Let $S$ be an open abelian subgroup and $L$ a left transversal for $S$ in $G$.  Then
\[
T=\bigcap_{b\in L}bSb^{-1}
\]
is an open normal subgroup.  Indeed, $L$ is finite and $T$ is an intersection of finitely may open
subgroups.  Furthermore, the definition of $T$ is independent of choice of transversal and for any $a$
in $G$, $aL$ is another transversal, hence $aTa^{-1}=\bigcap_{b\in aL}bSb^{-1}=T$.

Now if $a$ in $G$ and $t$ in $T$ are fixed, then for $s$ in $T$ we have
\[
sats^{-1}=a(a^{-1}sa)ts^{-1}=ats^{-1}(a^{-1}sa).
\]
Since $T$ is abelian and compact we have that
\[
R_a=\{s^{-1}a^{-1}sa:s\in T\}
\]
is a closed subgroup of $T$.  Observe that for $a$ and $a'$ in $G$, if $Ta=Ta'$, then
$R_a=R_{a'}$. Hence we may write $R_{Ta}=R_{aT}=R_a$.
We recall that $\fal(G)=\bigoplus_{aT\in G/T}\fal(aT)$ where $\fal(aT)=1_{aT}\fal(G)\cong\fal(T)$
and we obtain the desired form for $\cent_T\fal(G)$ and its spectrum $X=G/{\sim_T}$.  

It is evident that $\za(G)\subseteq \cent_T\fal(G)$, and the action of $G$ by inner automorphisms on
$G/{\sim_T}$ is really an action by $G/T$.  In fact of we let
$Z_Tu(x)=\int_T u(sxs^{-1})\,ds$, then the Weyl integral formula tells us that
$Z_G=Z_G\circ Z_T=Z_{G/T}\circ Z_T$.  Hence we gain the desired realization of $\za(G)$.

To see the action of $G$ on $X$, and hence the structure of the conjugacy class $C_{at}$,
we fix $a$ and $t$ as above, and for $b$ in $G$ and $s$ in $T$ we have
\begin{align*}
bab^{-1}&(btb^{-1})[bs^{-1}b^{-1}(ba^{-1}b^{-1}bsb^{-1}bab^{-1})] \\
&=bab^{-1}(ba^{-1}b^{-1}bsb^{-1}bab^{-1})(btb^{-1})bs^{-1}b^{-1}
=bs(at)(bs)^{-1}.
\end{align*}
Since each $bsb^{-1}$ is a generic element of $T$, we get the desired result.
\end{proof}

\begin{theorem}\label{theo:vahta}
If $G$ is virtually abelian, then $\za(G)$ is hyper-Tauberian and amenable.
\end{theorem}

\begin{proof}
We consider the algebra $\cent_T\fal(G)$ and its spectrum $X$, whose form is desribed
in Theorem \ref{theo:va}.
Each $\fal(aT:R_a)\cong\fal(T/R_a)$ (which is the abelian group algebra
$\mathrm{L}^1(\what{T/R_a})$).  The diagonal $(T/R_a)_D$ in
$\fal(T/R_a)\hat{\otimes}\fal(T/R_a)$ is a subgroup and hence,
thanks to \cite{rosenthal}, hyper-Ditkin.  Thus Remark \ref{rem:wik}
shows us that $X_D\cong\bigcup_{aT\in G/T}(T/R_a)_D$ is hyper-Ditkin.
Letting $B=G/T$, we appeal to Corollary \ref{cor:zbafinaht}.
\end{proof}

It is surprisingly easy, from this point, to give a characterization of those
subsets $E$ of $\conj(G)$ which are approximable (as defined in the last section) for $\za(G)$.   
For abelian $G$ this was achieved
in \cite{liuvrw} and for amenable $G$ this was achieved in \cite{forrestkls}.  A coset
is any subset $K$ of $G$ which is closed under the ternary operation: $x,y,z\in K$ implies
$xy^{-1}z\in K$.  It is an exercise to see that this agrees with the ``standard" notion of coset
of some subgroup.  We let
$\Omega(G)$ denote the Boolean algebra generated by all cosets of $G$, and
$\Omega_c(G)$ denote those elements of $\Omega(G)$ which are closed.

\begin{proposition}\label{prop:virtabel1}
Let $G$ be compact, and $E$ be a closed subset of $\conj(G)$.  Then
$E$ is approximable for $\za(G)$ if and only if $\til{E}=\bigcup_{C\in E}C\in\Omega_c(G)$.
\end{proposition}

\begin{proof}
If $(u_\alp)$ is a bounded approximate identity for $\ideal_{\za}(E)$.  Then
$(u_\alp)$ is a bounded net in $\fal(G)\subseteq\mathrm{B}(G_d)$, where the latter space is 
the Fourier-Stieltjes algebra of the discretized group $G_d$.  The embedding is an isometry thanks
to \cite{eymard}.  Since $\za(G)$ is regular, $(u_\alp)$ 
converges pointwise to the indicator function $1_{\conj(G)\setminus E}$ on $\conj(G)$, hence to
$1_{G\setminus\til{E}}$ in the weak* topology of $\mathrm{B}(G_d)$.  Hence by \cite{host},
$\til{E}\in\Omega(G)$.  Since $\til{E}$, being the pre-image of $E$ in $G$ under the conjugation
equivalence, is closed, we gain the desired conclusion.

If $\til{E}\in\Omega_c(G)$, then by \cite{forrestkls}, $\til{E}$ is approximable for $\fal(G)$.
If $(v_\alp)$ is a bounded approximate identity for $\ideal_\fal(\til{E})$, then
$(Z_Gv_\alp)$ is such for $\ideal_{\za}(E)$.
\end{proof}

We remark that the last proposition  reduces the 
general question of amenability of $\za(G)$ into a question
of whether $\conj(G)_D$ is in $\Omega(G\times G)$. 
(Indeed it follows from Lemma \ref{lem:tensprod} that $\za(G)\hat{\otimes}\za(G)\cong
\za(G\times G)$).  We do not know how to determine this for a general, even totally disconnected, group,
unless it is a product of finite groups; see Theorem \ref{theo:prodfinite}, for example.

Assesing which conjugacy-closed subsets are in $\Omega_c(G)$ for virtually abelian $G$
is surprisingly straightforward.

\begin{proposition}\label{prop:virtabel2}
If $G$ is compact and
virtually abelian, and $E\in \Omega_c(G)$, then $G^*E=\bigcup_{x\in E}C_x\in\Omega_c(G)$.
\end{proposition}

\begin{proof}
A result of \cite{forrest} (after \cite{gilbert,schreiber}) shows that there are a finite
number of closed subgroups $H_1,\dots,H_n$, elements $a_1,\dots, a_n$ of $G$, and
for each $k$, open subgroups $K_{k1},\dots, K_{km_k}$ of $H_k$ and elements
$b_{k1},\dots,b_{km_k}$ of $H_k$ such that $E=\bigcup_{k=1}^na_k
\left(H_k\setminus\bigcup_{j=1}^{m_k}b_{kj}K_{kj}\right)$.  Since each $H_k$ is compact,
each $K_{kj}$ is of finite index.  We can hence rearrange this result to show that
$E$ is simply a union of finitely many closed cosets of subgroups of $G$.  Let such a coset
be given by $aH$ where $H$ is a closed subgroup.  By taking intersection, we may suppose
$H$ is a subgroup of an open normal abelian subgroup $T$, hence $aH\subset aT$.

However, the calculations from the proof of Theorem \ref{theo:va} show that the orbit of $aH$
under conjugation by $G$ is $\bigcup_{bT\in G/T}bab^{-1}bHb^{-1}R_{bab^{-1}}$, which
is clearly an element of $\Omega_c(G)$.
\end{proof}

\begin{remark}\label{rem:lasser}
Let $G$ and $T$ be as in Theorem \ref{theo:va}.
Using reasoning above, we see that $1_T\za(G)=\cent_{G/T}\fal(T)$, is amenable.
In particular for $G=\Tee\rtimes\{\id,\iota\}$, where $\iota(t)=t^{-1}$, we have
that $\cent_{\{\id,\iota\}}\fal(\Tee)\cong\cent_{\{\id,\hat{\iota}\}}\ell^1(\Zee)\cong\ell^1(\Zee/\{\id,\hat{\iota}\})$. 
Here, $\Zee/\{\id,\iota\}\cong\En_0$ is the polynomial hypergroup with multiplication
$\del_n\ast\del_m=\frac{1}{2}(\del_{|n-m|}+\del_{n+m})$.  This hypergroup algebra is
is also proved to be amenable in \cite{lasser0}.  

In fact, we may define a class of hypergroups by letting $F$ be any finite subgroup of 
$\mathrm{GL}_n(\Zee)$ and considering the orbit space $\Zee^n/F$.  We let
$\ell^1(\Zee^n/F)$ denote the closed subalgebra of $\ell^1(\Zee^n)$ generated by elements
\[
\del_{F(v)}=\frac{1}{|F|}\sum_{\alp\in F}\del_{\alp(v)},\;v\in\Zee^n.
\]
We have that $\ell^1(\Zee^n/F)\cong\cent_F\ell^1(\Zee^n)$ is amenable.
Indeed, if $G=\Tee^n\rtimes F$, where $F$ acts by dual action, then
$\cent_F\ell^1(\Zee^n)\cong\cent_{G/\Tee^n}(\Tee^n)$, and we appeal to
Theorem \ref{theo:vahta}.
\end{remark}

\begin{remark}\label{rem:zlone}  
Similarly as in Theorem \ref{theo:va}, by working on the dense subspace
of continuous functions, we may obtain a decomposition
\[
\cent\bl^1(G)\cong\cent_{G/T}\cent_T\bl^1(G)\cong\cent_{G/T}\left(
\bigoplus_{aT\in G/T}\del_a\ast\bl^1(T/R_{aT})\right).
\]
Hence if we could prove that $\cent_T\bl^1(G)$ is amenable, we will gain the amenability
of $\cent\bl^1(G)$ thanks to the main result of \cite{kepert}.
\end{remark}

\subsection{Hyper-Tauberian property and weak amenability}
We can give a complete characterizetion of both hyper-Tauberianness and weak amenability
for $\za(G)$.

\begin{theorem}\label{theo:zawa}
For any compact group $G$ the following are equivalent:

{\bf (i)} the connected component of the identity, $G_e$, is abelian;

{\bf (ii)} $\za(G)$ is hyper-Tauberian;

{\bf (iii)} all singleton sets of $\conj(G)$ are spectral for $\za(G)$;

{\bf (iv)} $\za(G)$ is weakly amenable; and

{\bf (v)} $\za(G)$ admits no bounded point derivations.
\end{theorem}

\begin{proof}
That (ii) implies (iii) and (iv) are both from \cite[Theo.\ 5]{samei}.  
A well-known observation from \cite{singerw} is that a commutative Banach algebra admits 
a bounded point derivation at a multiplicative functional $\chi$ if and only if $(\ker\chi)^2$
is dense in $\ker\chi$.  Hence (iii) implies (v).  That (iv) implies (v) follows from a well-known 
fact mentioned in Section \ref{sec:preliminaries}.
Theorem \ref{theo:nacc} provides that (v) implies (i).

Hence it remains to see that (i) implies (ii).  In the case that $G$ is virtually abelian, this is from
Theorem \ref{theo:vahta}.

Now let us consider the general case where $G_e$ is abelian.  We follow the proof
of \cite[Theo.\ 3.3]{forrestr}.  We let $\fN$ be a net of closed normal subgroups, ordered by reverse 
inclusion, such that for any neighbourhood $U$ of $e$, we eventually have $N\subset U$ for 
some $N$ in $\fN$, and for which $G/N$ is Lie for each $N$ in $\fN$. For example, we may let
$N=N_F=\bigcap_{\pi\in F}\ker\pi$ for the increasing net of all finite subsets of $\what{G}$.
Then each $G/N$ is virtually abelian, which follows from \cite[(7.12)]{hewittrI}, for example. 
Hence, by  Proposition \ref{prop:quotient}, $\za(G/N)\cong P_N(\za(G))$, and hence is 
hyper-Tauberian.  Also if $N\supset N'$ then $P_N(\za(G))\subset P_{N'}(\za(G))$.
It then is easy to check that for $u$ in $\za(G)$, the difference $u-P_Nu$ tends to $0$ as $N$ tends to $e$.
Hence $\bigcup_{N\in\fN}P_N\za(G)$ is dense in $\za(G)$.  We appeal to \cite[Cor.\ 13]{samei} to
see that $\ell^1\text{-}\bigoplus_{N\in\fN}P_N\za(G)$ is hyper-Tauberian, and hence
the completion $\za_\fN(G)$ of $\bigcup_{N\in\fN}P_N\za(G)$, with respect to the norm
\[
\norm{u}_\fN=\inf\left\{\sum_{N\in\fN}\norm{u_N}_\fal:u=\sum_{N\in\fN}u_N,u_N\in P_N\za(G)\right\}
\geq\norm{u}_\fal
\]
is hyper-Tauberian, thanks to \cite[Theo.\ 12]{samei}.  Notice for $N\supseteq N'$ that
$P_N\za(G) P_{N'}\za(G)\subseteq P_{N'}\za(G)$, so $\norm{\cdot}_\fN$ is indeed an algebra norm.
But the continuous inclusion with dense range, $\za_\fN(G)\hookrightarrow\za(G)$, shows again
by  \cite[Theo.\ 12]{samei} that the latter algebra is hyper-Tauberian.
\end{proof}

\begin{remark}
{\bf (i)} We note that using the density of $\bigcup_{N\in\fN}P_N\za(G)$ in $\za(G)$, above,
we may have more easily shown that (i) implies (iv) directly.  We found it more satisfying
to obtain the stronger hyper-Tauberian property.  We note that \cite[Rem.\ 24(ii)]{samei} shows that
hyper-Tauberianness is stronger than weak amenability.

{\bf (ii)} A technique employed in the proof of \cite[Lem.\ 3.6]{forrestss1} and \cite[Theo.\ 2.4]{azimifardss}
can be used to allow us to bypass the algebra $\za_\fN(G)$, employed above.  Our present technique
allows us to avoid introducing local maps, which, admittedly, are used in the original definition
of hyper-Tauberianess in \cite{samei}. 
\end{remark}

The following partially recovers a non-spectral result of \cite{meaney}, but uses different methods.
We defined weak spectrality at the end of the last section.

\begin{proposition}\label{prop:finitesynthweaksynth}
Let $G$ be a non-abelian compact connected Lie group.  

{\bf (i)} There exists a $C$ in $\conj(G)$
for which $C$ is not spectral for $\fal(G)$.  

{\bf (ii)} Any finite $F\subseteq\conj(G)$ is weakly spectral for $\za(G)$ with
$\xi_{\za}(F)\leq |F|+\sum_{C\in F}\dim C/2$ where $\dim C$ is the dimension of the manifold $C$.
\end{proposition}

\begin{proof}
The failure of spectrality for some $\{C\}$ follows from (iii), above, and Proposition \ref{prop:weaksyn}.
On the other hand, \cite[Cor.\ 4.9]{ludwigt} shows that a single conjugacy class
$C$ is always weakly spectral for $\fal(G)$
with $\xi_\fal(C) \leq1+\dim C/2$.  Hence, a result in \cite{warner} shows that
\[
\xi_\fal\left(\bigcup_{C\in F}C\right)\leq\sum_{C\in F}\xi_\fal(C)=|F|+\sum_{C\in F}\dim C/2.
\]
Again we appeal to Proposition \ref{prop:weaksyn}.
\end{proof}

\section{Finite groups and their direct products}\label{sec:amenconst} 

We recall from \cite{johnson1}, that amenability of a Banach algebra $\fA$
is equivalent to having a {\it bounded approximate diagonal} (b.a.d.):  a bounded net
$(d_\iota)\subset\fA\hat{\otimes}\fA$ which for each $a$ in $\fA$ satisfies
\[
(a\otimes 1)d_\iota-d_\iota(1\otimes a)\overset{\iota}{\longrightarrow}0\text{ and }
m(d_\iota)a\overset{\iota}{\longrightarrow}a
\]
where $m:\fA\hat{\otimes}\fA\to\fA$ is the multiplication map. 
We let the {\it amenability constant} be given by
\[
\amen(\fA)=\inf\{M>0:\text{there is a b.a.d.\ }(d_\iota)\text{ for }\fA\text{ with }
\norm{d_\iota}_{\fA\hat{\otimes}\fA}\leq M\}
\]
where we adopt the convention that $\inf\varnothing=\infty$.
It will be useful for us to understand the tensor product $\za(G)\hat{\otimes}\za(G')$,
where $G'$ is another compact group.

\begin{lemma}\label{lem:tensprod}
We have an isometric isomorphism 
\[
\za(G)\hat{\otimes}\za(G')\cong\za(G\times G').
\]
\end{lemma}

\begin{proof}
Let us give two proofs.  

For the first, we recall the theorem of \cite{effrosr} that
\begin{equation}\label{eq:effrosr}
\fal(G)\hat{\otimes}^{op}\fal(G')\cong\fal(G\times G')
\end{equation}
where $\hat{\otimes}^{op}$ denotes the operator projective tensor product.  The map 
$Z_G:\fal(G)\to\za(G)$ is easily verified to be a complete quotient map, so we have a 
completley isometric inclusion
\[
\za(G)\hat{\otimes}^{op}\za(G')=Z_G\otimes Z_{G'}(\fal(G)\hat{\otimes}^{op}\fal(G'))
\subset\fal(G)\hat{\otimes}^{op}\fal(G').
\]
But  in the identification (\ref{eq:effrosr}), we have that $Z_G\otimes Z_{G'}\cong Z_{G\times G'}$,
so $Z_G\otimes Z_{G'}(\fal(G)\hat{\otimes}^{op}\fal(G'))\cong\za(G\times G')$.  Since
$\za(G)^*\cong\cent\vn(G')$ is a commutative von Neuman algebra, we obtain an isometric identification
\linebreak $\za(G)\hat{\otimes}^{op}\za(G')\cong\za(G)\hat{\otimes}\za(G')$.

For the second proof, we use the fact that $\za(G)\cong\ell^1(\what{G},d^2)$ as noted in Section
\ref{sec:preliminaries}.  Using that $\what{G\times G'}\cong\what{G}\times\what{G'}$ (irreducible
representations of products are exactly the Kroenecker products of irreducible representations)
we see that $\za(G\times G')\cong\ell^1(\what{G}\times\what{G'},d^2\times d^2)$, where
$d^2\times d^2(\pi,\pi')=d_\pi^2 d_{\pi'}^2$.  Hence the usual tensor product formula
shows that
\[
\za(G)\hat{\otimes}\za(G')\cong\ell^1(\what{G},d^2)\hat{\otimes}\ell^1(\what{G'},d^2)\cong
\ell^1(\what{G}\times\what{G'},d^2\times d^2)\cong\za(G\times G')
\]
with isometric identifications.
\end{proof}

The following computation on a finite group mirrors \cite[Theo.\ 1.8]{azimifardss}, where
for a finite group $G$ it is shown that
\[
\amen(\cent\bl^1(G))=\frac{1}{|G|^2}\sum_{C,C'\in\conj(G)}
|C||C'|\left|\sum_{\pi\in\what{G}}d_\pi^2\chi_\pi(C)\wbar{\chi_\pi(C')}\right|.
\]

\begin{proposition}\label{prop:amenconst}
Let $G$ be a finite group.  Then
\[
\amen(\za(G))=\frac{1}{|G|^2}\sum_{\pi,\pi'\in\what{G}\times\what{G}}d_\pi d_{\pi'}
\left|\sum_{C\in\conj(G)}|C|^2\chi_\pi(C)\wbar{\chi_{\pi'}(C)}\right|.
\]
In particular we see that $1\leq \amen(\za(G))$, with the bound achieved exactly
when $G$ is abelian.
\end{proposition}

\begin{proof}
Any bounded approximate diagonal admits a cluster point, which is a diagonal; i.e.\ $d$ in
$\za(G)\hat{\otimes}\za(G)$ for which $m(d)=1$ and $(u\otimes 1)d=d(1\otimes u)$.
It was observed in \cite{ghandeharihs} that for a finite dimensional amenable commutative algebra
that the diagonal is unique.  In fact Lemma \ref{lem:tensprod} provides that
this diagonal must be the indicator function of the diagonal
of the spectrum of $\za(G\times G)$, $1_{\conj(G)_D}=\sum_{C\in\conj(G)}1_{C\times C}$.
The Schur orthogonality relations provide Fourier series
\begin{align*}
1_{C\times C}
&=\sum_{\pi,\pi'\in\what{G}\times\what{G}}\langle 1_{C\times C}|\chi_\pi\otimes\chi_{\pi'}\rangle
\chi_\pi\otimes\chi_{\pi'} \\
&=\sum_{\pi,\pi'\in\what{G}\times\what{G}}\left(\frac{1}{|G|^2}
\sum_{x,y\in G\times G}1_{C\times C}(x,y)\wbar{\chi_\pi(x)\chi_{\pi'}(y)}\right)
\chi_\pi\otimes\chi_{\pi'} \\
&=\frac{1}{|G|^2}\sum_{\pi,\pi'\in\what{G}\times\what{G}}|C|^2\wbar{\chi_\pi(C)\chi_{\pi'}(C)}
\chi_\pi\otimes\chi_{\pi'}
\end{align*}
and hence
\[
1_{\conj(G)_D}=\frac{1}{|G|^2}\sum_{\pi,\pi'\in\what{G}\times\what{G}}\left(\sum_{C\in\conj(G)}
|C|^2\chi_\pi(C)\wbar{\chi_{\pi'}(C)}\right)\chi_{\bar{\pi}}\otimes\chi_{\pi'}
\]
where we have exchanged $\bar{\pi}$ for $\pi$ to give our formula its ``positive-definite" flavour.
We again appeal to Lemma \ref{lem:tensprod} and obtain
\[
\amen(\za(G))=\norm{1_{\conj(G)_D}}_{\za(G\times G)} 
\]
which, by (\ref{eq:zanorm}) gives us the desired result.  

Let us examine the lower bound.  We restrict the outer sum to the diagonal to obtain
\begin{align*}
\amen(\za(G))&
\geq\frac{1}{|G|^2}\sum_{\pi\in\what{G}}d_\pi^2\sum_{C\in\conj(G)}|C|^2\chi_\pi(C)\wbar{\chi_\pi(C)} \\
&\geq \frac{1}{|G|}\sum_{\pi\in\what{G}}d_\pi^2
\sum_{C\in\conj(G)}\frac{|C|}{|G|}\chi_\pi(C)\wbar{\chi_\pi(C)} \\
&=\frac{1}{|G|}\sum_{\pi\in\what{G}}d_\pi^2\langle\chi_\pi|\chi_\pi\rangle=
\frac{1}{|G|}\sum_{\pi\in\what{G}}d_\pi^2=1
\end{align*}
Notice that if $G$ is non-abelian, then at least one conjugacy class satisfies $|C|^2>|C|$, and 
the second inequality, above, is strict.  For an abelian group, $\za(G)=\fal(G)\cong\bl^1(\what{G})$.
The well-known diagonal $\frac{1}{|\what{G}|}\sum_{\chi\in\what{G}}\del_{\bar{\chi}}\otimes\del_\chi$
shows that $\amen(\bl^1(\what{G}))=1$.  \end{proof}

For a finite non-abelian group, the lower bound of $\amen(\cent\bl^1(G))\geq 1+\frac{1}{300}$
was derived based on a result in \cite{rider}.  Some improvements were made in the investigation
\cite{alaghmandancs}; however no better general lower bound, valid for all non-abelian finite
groups, is known to the authors.  Hence the following result stands as a pleasant contrast.

\begin{corollary}\label{cor:stan}
If $G$ is a non-abelian finite group, then $\amen(\za(G))\geq\frac{2}{\sqrt{3}}$.
\end{corollary}

\begin{proof}
Because $G$ is compact, we have that $\fal(G)$ is its own multiplier algebra, even its
own completley bounded multiplier algebra.  As such  each $u$ in $\fal(G)$ induces
a Schur multiplier on $G\times G$ matrices, $[a_{st}]\mapsto [u(s^{-1}t)a_{st}]$, with
norm the same as $\norm{u}_\fal$.  See \cite{bozejkof,jolissaint} for details of this.

The reasoning above also applies to $\fal(G\times G)$.  The diagonal $w=1_{\conj(G)_D}$
is an element of $\za(G\times G)\subset\fal(G\times G)$ is an idempotent, i.e.\ $ w^2=w$, and
$\norm{w}_\fal>1$.  Hence by \cite[Theo.\ 3.3]{stan} (using estimates which go back to
\cite{livschits}), we have that $\amen(\za(G))=\norm{w}_\fal\geq\frac{2}{\sqrt{3}}$.
\end{proof}

\begin{lemma}\label{lem:amenprop}
{\bf (i)} If $G_1,\dots, G_n$ are finite groups and $P=\prod_{i=1}^nG_i$, then
\[
\amen(\za(P))=\prod_{i=1}^n\amen(\za(G_i)).
\]
{\bf (ii)} If $G=H\times F$ where $H$ is compact and $F$ is finite, then
\[
\amen(\za(G))\geq\amen(\za(F)).
\]
\end{lemma}

\begin{proof}
To see (i), we use Lemma \ref{lem:tensprod} and the isomorphism $P\times P\cong
\prod_{i=1}^nG_i\times G_i$ to see that 
\[
\za(G\times G)\cong\za(G_1\times G_1)\hat{\otimes}\dots\hat{\otimes}\za(G_1\times G_1).
\]
Hence the unique diagonal satisfies
\[
1_{\conj(G)_D}\cong1_{\conj(G_1)_D}\otimes\dots\otimes 1_{\conj(G_n)_D}.
\]
We appeal to the fact that $\hat{\otimes}$ gives a cross-norm.

To see (ii) we have that the map $u\otimes v\mapsto u(e)v:\za(H)\hat{\otimes}\za(F)\to\za(F)$
extends to a contractive surjective homomorphism, and hence, again using Lemma \ref{lem:tensprod},
iduces a contractive surjective homomorphism from $\za(G)$ onto $\za(F)$.  It is
standard and straighforward to check that if $\amen(\za(G))<\infty$, then any bounded
approximate diagonal for $\za(G)$ is carried to such for $\za(F)$, hence the diagonal
for $\za(F)$ has norm bounded above by $\amen(\za(G))$.
\end{proof}

We lend the following evidence to our conjecture that $\za(G)$ is amenable if and
only if $G$ is virtually abelian.

\begin{theorem}\label{theo:prodfinite}
Let $\{G_i\}_{i\in I}$ be an collection of finite groups and $P$ be the compact group
product group $\prod_{i\in I}G_i$.  Then $\za(P)$ is amenable if and only if 
all but finitely many groups $G_i$ are abelian.
\end{theorem}

\begin{proof}
Suppose there is an infinite sequence of indices $i_1,i_2,\dots$ for which
each $G_{i_k}$ is non-abelian.  Let $P_n=\prod_{k=1}^nG_{i_k}$ and $H_n
=\prod_{i\in I\setminus\{i_1,\dots,i_n\}}G_i$.
We successively use parts (ii) and (i) of the lemma above, then Corollary
\ref{cor:stan} to see for each $n$ that
\[
\amen(\za(P))\geq \amen(\za(P_n))=\prod_{k=1}^n\amen(\za(G_{i_k}))\geq
(2/\sqrt{3})^n.
\]
Thus we see that $\amen(\za(P))=\infty$.
\end{proof}

Using techniques from the theory of hypergroups, the first-named author (\cite{alaghmandan}) has
proved that if $G$ is {\it tall}, i.e.\ $\lim_{\pi\to\infty}d_\pi=\infty$, then $\za(G)$
is non-amenable.  There are examples of totally
disconnected tall groups in \cite{hutchinson}.  Coupled with the last theorem,
this gives two classes of totally disconnected and non-virtually abelian $G$ for which $\za(G)$ 
is non-amenable.

\section{Open questions}

In the course of this investigation, three open questions stand out.

\begin{question}
For compact $G$, does amenability of $\za(G)$ imply that $G$ is abelian?
\end{question}

An approach to answering this is suggested in comments following Proposition 
\ref{prop:virtabel1}. 
Thanks to Theorem \ref{theo:zawa}, this question remains open only for
on compact groups with abelina connected compoents of identity.   
Totally disconnected compact groups are pro-finite, and hence 
more refined qualitative version of Corollary \ref{cor:stan}, coupled with
Proposition \ref{prop:quotient} may solve this.  

For the next question, we use the assumptions and
notation of Section \ref{ssec:zba}.  We state it in two equivalent forms.

\begin{question}\label{q:za} {\rm \cite{spronk}} 
{\bf (i)}  If $\fA$ is amenable, must it be hyper-Tauberian?

{\bf (ii)}  If $X_D$ is approximable for $\fA\hat{\otimes}\fA$, must it be spectral?
\end{question}

As indicated in Remark \ref{rem:belcherr}, its is not generally true that
approximability of a subset of the spectrum, for an a given algebra, implies spectrality.

The next question motivates much of the present research, and is scantly addressed
in Remark \ref{rem:zlone}.

\begin{question}\label{q:zlone} {\rm \cite{azimifardss}}
For a compact group $G$, is it true that $\cent\bl^1(G)$ is amenable if and only if
$G$ is virtually abelian?
\end{question}

\begin{center}{\sc Acknowledgements}
\end{center}

The authors would like to thank the Fields Institute and the organizers of the
{\it Thematic Program on Abstract Harmonic Analysis, Banach and Operator Algebras},
for hospitality and support during the period in which much of this work was conducted.
The first named author was a Fields Post-doctoral Fellow for the period of that program, while 
the second a visiting participant.  The second named author would
like thank NSERC Grant 312515-2010.

The authors are grateful to Y. Choi and E. Samei, who were the first author's doctoral advisors
at the University of Saskatchewan, and each of whom offered advice to both authors during the writing
of this paper.  Some of the results from Section \ref{sec:pointder} first
appeared in that author's dissertation.



Former address of M. Alaghmandan: \\ {\sc 
Fields Institute for Research in Mathematical Sciences, 222 College St., Toronto, ON, M5T 3J1,
Canada.} 

Current address of both authors: \\
{\sc Department of Pure Mathematics, University of Waterloo,
Waterloo, ON, N2L 3G1, Canada.}

\medskip
Email-adresses:
\linebreak
{\tt mahmood.a@uwaterloo.ca, nspronk@uwaterloo.ca}


\begin{thebibliography}{30}

\bibitem{alaghmandan}
M. Alaghmandan.
\newblock Amenability notions of hypergroups and some applications to locally compact groups.
\newblock Preprint, see {\tt arXiv:1402.2263}.

\bibitem{alaghmandancs}
M. Alaghmandan, Y. Choi and E. Samei.
\newblock ZL-amenability constants of finite groups with two character degrees.
\newblock To appear in Canad. Math. Bull., see {\tt arXiv:1302.1929}.



\bibitem{azimifardss}
A. Azimifard, E. Samei and N. Spronk.
\newblock Amenability properties of the centres of group algebras. 
\newblock J. Funct. Anal. 256 (2009), no.\ 5, 154--1564.

\bibitem{badecd}
W.G. Bade, P.C. Curtis and H.G. Dales.
\newblock Amenability and weak amenability for Beurling and Lipschitz algebras. 
\newblock Proc. London Math. Soc. (3) 55 (1987), no.\ 2, 359--377.


\bibitem{heyer}
W.R. Bloom and H. Heyer.
\newblock Harmonic analysis of probability measures on hypergroups. 
\newblock  de Gruyter, Berlin, 1995.

\bibitem{blecherr}
D.P. Blecher and C.J. Read.
\newblock Operator algebras with contractive approximate identities: A large operator algebra in $c_0$.
\newblock Preprint {\tt arXiv:1309.2023}.

\bibitem{bozejkof}
M. Bo\.{z}ejko and G. Fendler.
\newblock Herz-Schur multipliers and completely bounded multipliers of the Fourier algebra of a locally compact group.
\newblock  Bollettino U.M.I. (6) 3-A (1984), 297--302.

\bibitem{chilanar}
A.K. Chilana and K.A. Ross.  
\newblock Spectral synthesis in hypergroups. 
\newblock Pacific J. Math. 76 (1978), no.\ 2, 313--328. 

\bibitem{curtisl}
P.C. Curtis and R.J. Loy.  
\newblock The structure of amenable Banach algebras. 
\newblock  J. London Math. Soc. (2) 40 (1989), no.\ 1, 89--104. 

\bibitem{effrosr}
E.G. Effros and Z.-J. Ruan.
\newblock On approximation properties for operator spaces. 
\newblock Internat. J. Math. 1 (1990), no.\ 2, 163--187. 

\bibitem{eymard}
P. Eymard.
\newblock L'alg\`{e}bre de Fourier d'un groupe localement compact. 
\newblock  Bull. Soc. Math. France 92 (1964) 181--236.



\bibitem{forrest}
B.E. Forrest.  
\newblock  Amenability and bounded approximate identities in ideals of $A(G)$. 
\newblock Illinois J. Math. 34 (1990), no. 1, 1--25.

\bibitem{forrest0}
B.E. Forrest.
\newblock  Fourier analysis on coset spaces. 
\newblock  Rocky Mountain J. Math. 28 (1998), no.\ 1, 173--190. 

\bibitem{forrestkls}
B.E. Forrest, E. Kaniuth, A.T. Lau and N. Spronk.
\newblock  Ideals with bounded approximate identities in Fourier algebras. 
\newblock  J. Funct. Anal. 203 (2003), no.\ 1, 286--304.

\bibitem{forrestr}
B.E. Forrest and V. Runde.  
\newblock  Amenability and weak amenability of the Fourier algebra. 
\newblock  Math. Z. 250 (2005), no.\ 4, 731--744.

\bibitem{forrestss}
B.E. Forrest, E. Samei and N. Spronk.
\newblock  Weak amenability of Fourier algebras on compact groups. 
\newblock Indiana Univ. Math. J. 58 (2009), no.\ 3, 1379--1393.

\bibitem{forrestss1}
B.E. Forrest, E. Samei and N. Spronk.
\newblock Convolutions on compact groups and Fourier algebras of coset spaces. 
\newblock Studia Math. 196 (2010), no.\ 3, 223--249.

\bibitem{ghandeharihs}
M. Ghandehari, H. Hatami and N. Spronk.
\newblock  Amenability constants for semilattice algebras. 
\newblock  Semigroup Forum 79 (2009), no.\ 2, 279--297.

\bibitem{gilbert}
J.E. Gilbert.  
\newblock  On projections of $L^\infty(G)$ onto translation-invariant subspaces.
\newblock   Proc. London Math. Soc. 19 (1969) 69--88.

\bibitem{helemskii}
A.Ya.\ Helemskii.
\newblock The homology of Banach and topological algebras.
\newblock Translated from the Russian by Alan West. Mathematics and its Applications 
(Soviet Series), 41. Kluwer Academic Publishers Group, Dordrecht, 1989.


\bibitem{herz}
C. Herz.
\newblock  Harmonic synthesis for subgroups. 
\newblock  Ann. Inst. Fourier (Grenoble) 23 (1973), no.\ 3, 91--123. 

\bibitem{hewittrI}
E. Hewitt and K.A. Ross.
\newblock  Abstract harmonic analysis. Vol. I: Structure of topological groups. Integration theory, group representations. 
\newblock  Springer, New York, 1963

\bibitem{hewittrII}
E. Hewitt and K.A. Ross.
\newblock  Abstract harmonic analysis. Vol. II: Structure and analysis for compact groups. Analysis on locally compact Abelian groups. 
\newblock  Springer, New York, 1970.

\bibitem{host}
B. Host.  
\newblock  Le th\'{e}or\`{e}me des idempotents dans $B(G)$.
\newblock   Bull. Soc. Math. France 114 (1986) 215--223.

\bibitem{hutchinson}
M.H. Hutchinson.
\newblock  Tall profinite groups. 
\newblock  Bull. Austral. Math. Soc. 18 (1978), no.\ 3, 421--428.

\bibitem{jewett}
R.I. Jewett.
\newblock  Spaces with an abstract convolution of measures. 
\newblock  Advances in Math. 18 (1975), no.\ 1, 1--101.

\bibitem{johnson}
B.E. Johnson.
\newblock Cohomology in Banach algebras. 
\newblock Memoirs of the American Mathematical Society, No.\ 127, 1972.

\bibitem{johnson1}
B.E. Johnson.
\newblock Approximate diagonals and cohomology of certain annihilator Banach algebras. 
\newblock Amer. J. Math. 94 (1972), 685--698.

\bibitem{johnson94}
B.E. Johnson.
\newblock Weak amenability of group algebras. 
\newblock Bull. London Math. Soc. 23 (1991), no. 3, 281--284. 

\bibitem{jolissaint}
P. Joilissaint.
\newblock A characterization of completely bounded multipliers of Fourier algebras. 
\newblock Colloq. Math. 63 (1992), 311--313.

\bibitem{kaniuthB}
E. Kaniuth.
\newblock A course in commutative Banach algebras. 
\newblock  Springer, New York, 2009.

\bibitem{kepert}
A. Kepert.
\newblock Amenability in group algebras and Banach algebras. 
\newblock Math. Scand. 74 (1994), no.\ 2, 275--292. 

\bibitem{lasser0}
R. Lasser.
\newblock Amenability and weak amenability of $l^1$-algebras of polynomial hypergroups. 
\newblock Studia Math. 182 (2007), no.\ 2, 18--196. 


\bibitem{lasser}
R. Lasser.
\newblock Point derivations on the $L^1$-algebra of polynomial hypergroups. 
\newblock Colloq. Math. 116 (2009), no.\ 1, 15--30.


\bibitem{leelss}
H.H. Lee, J. Ludwig, E. Samei and N. Spronk.
\bibitem Weak amenability of Fourier algebras and local synthesis of the anti-diagonal
\bibitem Preprint, see {\tt arXiv:1502.05214}.

\bibitem{liuvrw}
T.S. Liu, A. van Rooij and J. Wang.
\newblock Projections and approximate identities for ideals in group algebras.
\newblock Trans. Amer. Math. Soc. 175 (1973) 469--482.

\bibitem{livschits}
L. Livshits.
\newblock A note on 0-1 Schur multipliers.
\newblock  Linear Algebra Appl. 222 (1995), 15--22.


\bibitem{ludwigst}
J. Ludwig, N. Spronk and L. Turowska.
\newblock  Beurling-Fourier algebras on compact groups: spectral theory. 
\newblock  J. Funct. Anal. 262 (2012), no.\ 2, 463--499.

\bibitem{ludwigt}
J. Ludwig and L. Turowska.
\newblock  Growth and smooth spectral synthesis in the Fourier algebras of Lie groups. 
\newblock  Studia Math. 176 (2006), no.\ 2, 139--158.

\bibitem{meaney}
C. Meaney.
\newblock On the failure of spectral synthesis for compact semisimple Lie groups. 
\newblock J. Funct. Anal. 48 (1982), no.\ 1, 43--57. 

\bibitem{price}
J.F. Price.
\newblock Lie groups and compact groups. 
\newblock Cambridge Univ.\ Press, 1977.

\bibitem{ruan}
Z.-J. Ruan.
\newblock The operator amenability of A(G). 
\newblock Amer. J. Math. 117 (1995), no. 6, 1449--1474. 

\bibitem{stegemanr}
H. Reiter and J.D. Stegeman.
\newblock Classical harmonic analysis and locally compact groups. 
Second edition. 
\newblock Oxford University Press, New York, 2000.

\bibitem{ricci}
F. Ricci.  
\newblock Local properties of the central Wiener algebra on the regular set of a compact Lie group. 
\newblock  Bull. Sci. Math. (2) 101 (1977), no.\ 1, 87--95. 

\bibitem{rider}
D. Rider.
\newblock Central idempotent measures on compact groups. 
\newblock Trans. Amer. Math. Soc. 186 (1973), 459--479.

\bibitem{rosenthal}
H.P. Rosenthal.
\newblock On the existence of approximate identities in ideals of group algebras. 
\newblock Ark. Mat. 7 (1967) 185--191.

\bibitem{samei0}
E. Samei.  
\newblock Bounded and completely bounded local derivations from certain commutative 
semisimple Banach algebras. 
\newblock Proc. Amer. Math. Soc. 133 (2005), no. 1, 229--238.

\bibitem{samei}
E. Samei.  
\newblock Hyper-Tauberian algebras and weak amenability of Fig\`{a}-Talamanca--Herz algebras. 
\newblock  J. Funct. Anal. 231 (2006), no.\ 1, 195--220.

\bibitem{schreiber}
B.M. Schreiber.
\newblock  On the coset ring and strong Ditkin sets
\newblock   Pacific J. Math. 32 (1970) 805--812.

\bibitem{singerw}
I.M. Singer and J. Wermer.
\newblock Derivations on commutative normed algebras. 
\newblock Math. Ann. 129 (1955) 260--264. 

\bibitem{spronk0}
N. Spronk.
\newblock Operator weak amenability of the Fourier algebra.  
\newblock Proc. Amer. Math. Soc. 130 (2002), no. 12, 3609--3617.

\bibitem{spronk}
N. Spronk.
\newblock Amenability properties of Fourier algebras and Fourier-Stieltjes algebras: a survey. 
\newblock Banach algebras 2009, 365--383, Banach Center Publ., 91, Polish Acad.\ Sci.\ Inst.\ Math., Warsaw, 2010.

\bibitem{stan}
A.-M.P. Stan.
\newblock On idempotents of completely bounded multipliers of the Fourier algebra $A(G)$. 
\newblock Indiana Univ. Math. J. 58 (2009), no.\ 2, 523--535.

\bibitem{tomiyama}
J. Tomiyama.
\newblock Tensor products of commutative Banach algebras. 
\newblock T\^{o}hoku Math. J. (2) 12 (1960) 147--154.

\bibitem{warner0}
C.R. Warner.
\newblock  A class of spectral sets. 
\newblock  Proc. Amer. Math. Soc. 57 (1976), no.\ 1, 9--102. 

\bibitem{warner}
C.R. Warner.
\newblock  Weak spectral synthesis.
\newblock Proc. Amer. Math. Soc. 99 (1987), no.\ 2, 24--248. 

\bibitem{wik}
I. Wik.  
\newblock  A strong form of spectral synthesis. 
\newblock  Ark. Mat. 6 (1965) 55--64.

\end{thebibliography}
\end{document}